\documentclass[a4paper,12pt]{amsart}

\usepackage{amscd,amssymb,amsmath,latexsym,graphicx,epsfig,color,url}
\usepackage[active]{srcltx}

\usepackage[english, algosection, algoruled, noline]{algorithm2e}

\makeatletter
\def\addots{\mathinner{\mkern1mu\raise\p@
\vbox{\kern7\p@\hbox{.}}\mkern2mu
\raise4\p@\hbox{.}\mkern2mu\raise7\p@\hbox{.}\mkern1mu}}
\makeatother

\catcode`\<=\active \def<{
\fontencoding{T1}\selectfont\symbol{60}\fontencoding{\encodingdefault}}
\newcommand{\assign}{:=}

\newcommand{\xx}{{\bf x}}
\newcommand{\R}{{\mathbb R}}
\newcommand{\C}{{\mathbb C}}
\newcommand{\Q}{{\mathbb Q}}
\newcommand{\N}{{\mathbb N}}

\newcommand{\diag}{{\rm diag}}

\def\ib{\mathbf{i}\,}

\newcommand{\vb}[1]{\mathbf{#1}}
\newcommand{\sprod}[1]{\langle{#1}\rangle}

\def\xx{\vb{x}}

\def\cC{\mathcal{C}}

\def\cQ{\mathcal{Q}}
\newcommand{\Int}{\mathrm{Int}}

\usepackage{xcolor}

\theoremstyle{definition}
\newtheorem{definition}{Definition}[section]
\newtheorem{notation}[definition]{Notation}
\newtheorem{remark}[definition]{Remark}
\newtheorem{example}[definition]{Example}

\theoremstyle{plain}
\newtheorem{lemma}[definition]{Lemma}
\newtheorem{proposition}[definition]{Proposition}

\newtheorem{corollary}[definition]{Corollary}

\title{Univariate Rational Sums of Squares}

\author{Teresa Krick} 
\author{Bernard Mourrain}
\author{Agnes Szanto}
\thanks{The research of Teresa Krick was partly supported by CONICET PIP-11220130100073CO and BID-PICT 2018-02315. The work of Bernard Mourrain was partly supported by the European Union’s Horizon 2020 research and innovation programme under the Marie Sk{\l}odowska-Curie Actions, grant agreement 813211 (POEMA). The research of Agnes Szanto was partly supported by NSF grant  CCF-1813340. }

\address{Teresa  Krick, Departamento de Matem\'atica \& IMAS, Universidad de Buenos Aires \& CONICET, Argentina.}\email{krick@dm.uba.ar}

\address{Bernard Mourrain, Aromath,
 Inria d'Universit\'e C\^ote d'Azur,
   2004, route des Lucioles,
   06902 Sophia Antipolis,
   France.
 }
\email{bernard.mourrain@inria.fr}

\address{Agnes Szanto, North carolina State University, Department of Mathematics, Campus Box 8205, Raleigh, NC, 27695, USA.}
\email{aszanto@ncsu.edu}

\begin{document}

\maketitle
\begin{center}  \em {To the memory of our beloved friend Agnes}\end{center}

\smallskip
\begin{abstract}
   Given rational univariate polynomials $f$ and $g$ such that $\gcd (f,g)$ and $f/\gcd(f,g)$ are relatively prime, we show that $g$ is non-negative at all the real roots of $f$ if and only if $g$ is a sum of squares of rational polynomials modulo $f$. We complete our study by exhibiting an algorithm that produces a  certificate that a polynomial $g$ is non-negative at the real roots of a non-zero polynomial $f$, when the above assumption is satisfied.
\end{abstract}

\section{Introduction}

It is a classical result  that a real univariate polynomial is {\em non-negative on all $\R$} if and only if it is a sum of  squares of real polynomials (and in fact, 2 polynomials are enough). It was then proved by Landau in 1905, see  \cite{landau}, that every univariate polynomial with {\em rational coefficients} which is {non-negative on all $\R$}  is a sum of 8 squares of {\em rational polynomials} (this result was  improved in \cite{pourchet}, lowering the bound  of 8  to the optimal value of 5).

We call this the {\em global case}, when we consider non-negativity on all $\R$. The {\em local case} is when we consider analogous  questions for a polynomial which is non-negative at the real roots of another non-zero polynomial. More explicitly,  the corresponding statement  is: Given  a non-zero   polynomial $f\in \R[x]$, is it true  that a polynomial $g\in \R[x]$  is {non-negative at all the real roots of $f$} if and only if it is congruent modulo $f$ to a sum of squares of polynomials in $\R[x]$? That is, if there exist polynomials $h_i\in \R[x]$, $1\le i\le N$ for some $N\in \N$, such that  
\begin{equation*}
h:=\sum_{i=1}^N {h}_i^2 \quad\mbox{satisfies} \quad h \equiv g \mod f.
\end{equation*}

In \cite{par1}, P. Parrilo gives a very simple construction that shows that this is indeed the case  in a  zero-dimensional radical  setting of multivariate polynomials. In our specific setting his result shows that every $g\in \R[x]$ which is non-negative at all the real roots of a {\em squarefree} polynomial $f\in\R[x]$  is congruent modulo $f$ to a sum of squares of real polynomials. 
In this paper, we consider  the corresponding {\em rational}  question: Given  polynomials $f, g\in \Q[x]$ such that $g$ is non-negative at all the real roots of $f$, is it true that $g$ is congruent modulo $f$ to a sum of squares of  polynomials $h_i\in \Q[x]$? Note that this is equivalent to say that $g$ is congruent modulo $f$ to a  {\em rational positive weighted sum} of squares of polynomials in $\Q[x]$, that is, that there exist $\omega_i\in \Q_{+}$ and  $h_i\in \Q[x]$, $1\le i\le N$, such that 
\begin{equation}\label{eq:SOS certificate}
h:=\sum_{i=1}^N {\omega}_i \,{h}_i^2 \quad\mbox{satisfies} \quad h \equiv g \mod f,
\end{equation}
since  for $\omega_i=m/n\in \Q$ with $m,n\in \N$, $\omega_i h_i^2= mn (h_i/n)^2$.

\smallskip
The positive weighted sum of squares $h$ is commonly called a {\em sum of squares (SOS) decomposition} of $g$ modulo $f$, and such a decomposition, together with the polynomial $q\in \Q[x]$ such that $g=h+q\,f$ is a {\em certificate} of   the non-negativity of $g$ at the real roots of $f$.

\smallskip 
The existence and computation of rational SOS decompositions of positive polynomials has been investigated in the univariate global case for instance in \cite{Chevillardetal2011},  \cite{MAGRON2019200}, or in the multivariate case in \cite{PeyrlParrilo}, \cite{KALTOFEN2008}.
A counter-example in \cite{scheiderer} shows that, in the multivariate case, a rational polynomial which is a sum of squares of real polynomials 
cannot always be decomposed as a  rational sum of squares.
In 
\cite{KALTOFEN2012},
\cite{GUO2012},
rational Artin's type certificates of positivity, that is, fractions of two rational weighted sums of squares polynomials are considered.
In \cite{MagronSafeyElDin2021}, algorithms to compute positivity certificates and bounds on their bit complexity and the size of their output are presented, 
including Artin's type certificates and rational weighted sums of squares certificates for positive polynomials on compact basic semi-algebraic sets.
The algorithms work under some strictly positivity assumptions.
They involve numeric-symbolic tools such as the perturbation algorithm of \cite{Chevillardetal2011},
the rounding-projection algorithm of \cite{PeyrlParrilo} or Semi-Definite Programming solvers.
{More recently, \cite{MagronExactOptimizationSums2019} provides 
a numeric-symbolic algorithm based on rounding-projection techniques for
computing exact representations of polynomials lying in the interior of the cone of nonnegative circuits (SONC) or of the cone of arithmetic-geometric-exponentials (SAGE). In \cite{MagronSumSquaresDecompositions2021}, an algorithm is proposed to compute the representation of a non-negative polynomial $f$ as a rational sum of squares and an element in the gradient ideal of $f$ with rational coefficients, under the hypothesis that the gradient ideal is zero-dimensional and radical, reducing to the univariate case by elimination techniques.
Numeric-symbolic approaches similar to \cite{MagronSafeyElDin2021} are applied to trigonometric polynomials in \cite{MagronExactSOHSdecompositions2022}.
}

\smallskip
In this paper, we first show, by a direct method, that a rational univariate polynomial $g$ 
strictly positive at  the real roots of a rational squarefree polynomial $f$ always admits  a rational SOS decomposition modulo $f$. This can be seen as a very very special case of Putinar's Theorem \cite{Putinar93} over \emph{rational} numbers. 
We then extend the result to rational univariate polynomials $g$ that are non-negative at the real roots of $f$, under an assumption specified in our main result:

\smallskip
\noindent {\bf Theorem.}  {\em Let $f\in \Q[x]$ be a non-zero   polynomial  of degree $n$ and $g\in \Q[x]$ be such that
$\gcd(f,g)$ and $f/\gcd(f,g)$ are relatively prime. Assume that 
$g$ is non-negative at all the real roots of $f$. Then there exist  rational positive weights  ${\omega}_i\in \Q_{+}$ and rational polynomials ${h}_i\in \Q[x]$ of degree $< n$, $1\le i \le N$ for some $N\in \N$, such that 
$$
h:=\sum_{i=1}^N{\omega}_i \,{h}_i^2 \quad\mbox{satisfies} \quad h \equiv g \mod f.
$$}

Note that when $f$ is squarefree, our assumption on $\gcd(f,g)$ and $f/\gcd(f,g)$ being relatively prime is automatically satisfied.
Furthermore, this assumption  seems to be  optimal in order for such an SOS decomposition to exist, as  the following example demonstrates \cite[Remark 1]{par1}:
For $f=x^2$ and $g=x$, $g$ is non-negative on all the (real) roots of $f$ but there is no such SOS decomposition. Note that in this case $\gcd(f,g)=x=f/\gcd(f,g)$ and the polynomials $f$ and $g$ do not satisfy the assumption of our theorem.

\smallskip
Certifying the non-negativity of a polynomial $g$ at the real roots of another polynomial $f$ is a problem of particular importance in Computer Algebra, for instance, for the localisation of real roots \cite{bpr-arag-03}, 
or in Automatic Theorem Proving for the certification of sign conditions over the real numbers. 
It is also useful for checking the sign of polynomials in $\R^n$ or more generally in Polynomial Optimization Problems, 
since one can generically add polynomial constraints like the gradient equations and reduce to a univariate polynomial sign certification problem by elimination of variables (see e.g. \cite{MagronSumSquaresDecompositions2021}).

\smallskip
The proof of our theorem is developed in 
Section~\ref{sec:exist}. It proceeds by first tackling in Subsection~\ref{sub:1} the case when $g$ is strictly positive at all the real roots of a squarefree polynomial $f$ of degree $n$: by modifying the construction in \cite{par1},  we first show there always exists a  {\em real} SOS decomposition $h$  of $g$ modulo $f$, $$h=[1,x,\dots,x^{n-1}] \,Q\, [1,x,\dots,x^{n-1}]^T$$  with $Q\in \R^{n\times n}$   symmetric and {\em  positive definite}. This enables us to perturb the real coefficients in matrix $Q$ in order to turn them rational, while keeping the condition of remaining an SOS decomposition for $g$ modulo $f$, as done in  \cite{PeyrlParrilo} for the global case (with the difference that here we know there always exists such a positive definite real matrix). In a second step, Subsection~\ref{sub:2} deals with the case of a non squarefree polynomial $f$, by applying Hensel lifting and Chinese Remainder Theorem recombination. We finally relax the strictly positive condition to non-negative under our assumption.  

\smallskip
In this paper, we also address the following algorithmic question: Can we produce an algorithm that computes a rational SOS certificate, which size is related to the geometry of the input polynomials?

Several algorithms can be used to certify that a polynomial $g$ is non-negative at the real roots of $f$. We refer to \cite{bpr-arag-03} for a general presentation of these algorithms, 
based for instance on Sturm-Habicth sequences
or isolation of real roots.
The algorithm that we describe in Section~\ref{sec:algorithm} does not require to   isolate or approximate the real roots of $f$. It computes 
a certificate of non-negativity by computing an SOS decomposition of $g$ modulo $f$ using two main ingredients.
The first ingredient is an adaptation of the rounding-projection algorithm of \cite{PeyrlParrilo} to the case of a rational polynomial $g$ strictly positive at the real roots of a squarefree polynomial $f$, following the proof of Proposition \ref{prop:existence}. 
The second ingredient is a reduction of the general case when $\gcd(f,g)$ and $f/\gcd(f,g)$ are relatively prime, to the strictly positive case, then lifting the rational SOS decompositions via Hensel lifting and Chinese Remainder Theorem, following the proof of our main theorem.

\medskip
\noindent {\bf Acknowledgment.} {This collaboration and research project started because Agnes  contacted the two other authors after an invitation by the  organizers of the MCA 2021 session ``Symbolic computation: theory, algorithms and applications", Alicia Dickenstein, Alexey Ovchinnikov and Veronika Pillwein,  to submit a publication related to her talk to the Revista de la Uni\'on Matem\'atica Argentina. We all worked together during 2021 and, as usual when working with her, Agnes' input was crucial to produce the output. Agnes  sadly passed away on March 21, 2022
. We miss her dearly.}

\section{Existence of a rational SOS decomposition}\label{sec:exist}

\subsection{The squarefree and strictly positive case} \label{sub:1}

In this section we assume that $f\in \R[x]$ is a squarefree polynomial and that $g\in \R[x]$ is strictly positive at the real roots of $f$. We fix the following notation.

\begin{notation} \label{not:f}
We denote 
$$f=\sum_{i=0}^n f_i \, x^i=f_n(x-\xi_1)\cdots (x-\xi_n) \ \mbox{ with } \ \xi_i\ne \xi_j\in \C \ \mbox{ for } \ i\ne j,$$
where   $\xi_1, \ldots, \xi_k$ are the real roots of $f$ (for some $0\le k\le n$) while  the complex non-real roots are labeled as 
$\xi_{k+2 i-1}, \xi_{k+2i}$ with $\xi_{k+2 i-1}= \overline{\xi_{k+2i}}$ for $1\le i\le  \frac{n-k}{ 2}$.
\end{notation}

The Lagrange basis for $\xi_1,\ldots,\xi_n$  is denoted by $u_1,...,u_n\in \C[x]$, i.e.
\begin{equation}\label{eq:lagrange0}u_i=\prod_{j\ne i} \frac{x-\xi_j}{\xi_i-\xi_j} = \frac{f}{f'(\xi_i)(x-\xi_i)}\quad \mbox{for } \, 1\le i\le n. \end{equation}
It satisfies that for any polynomial $p\in \C[x]$ one has
\begin{equation}\label{eq:interpolation}
    p(x) \equiv \sum_{i=1}^{n} p(\xi_i)\, u_i(x) \mod f.
\end{equation}
The  basis $u_1,\dots,u_n$ is also defined by the conditions  $\deg(u_i)\leq n-1$, $1\le i\le n$ and  $u_i(\xi_j)=\delta_{i,j}$ for $1\le i,j\le n$. This implies by \eqref{eq:interpolation} that
\begin{equation}\label{eq:lagrange}u_i^2\equiv u_i \mod f \ \mbox{ for } 1\le i\le n \ \mbox{ and } \ u_iu_j\equiv 0 \mod f \ \mbox{ for } i\neq j.
\end{equation}

Given $g\in \R[x]$, Parrilo constructed in \cite{par1} the following real polynomial 

\begin{align}\label{eq:parrilo pol}
 & \sum_{i=1}^k    g(\xi_i)\, u_i^2 + \sum_{i=1}^{(n-k)/2}\big( \sqrt{g(\xi_{k+2i})}u_{k+2i}+\sqrt{g(\overline{\xi_{k+2i}})}\overline{u_{k+2i}}\big)^2\\
 & = \sum_{i=1}^n    g(\xi_i)\, u_i^2 + 2\sum_{i=1}^{(n-k)/2} |g(\xi_{k+2i})|\,u_{k+2i-1} u_{k+2i},\nonumber
\end{align}
where the identity follows from the fact that  the interpolation polynomials associated to the complex non-real roots of $f$ are pairwise conjugate, i.e. 
$\overline{u_{k+2i}}=u_{k+2i-1}$.

This polynomial is  a sum of squares in $\R[x]$ whenever $g$ is non-negative at the real roots of $f$, as shown by identity \eqref{eq:parrilo pol},
since for $1\le i \le \frac{n-k}{2}$,
$$ \sqrt{g(\xi_{k+2i})}u_{k+2i}+\sqrt{g(\overline{\xi_{k+2i}})}\overline{u_{k+2i}}=2\Re( \sqrt{g(\xi_{k+2i})}u_{k+2i}),$$ where
$\Re$ denotes the real part.
Furthermore, it is congruent to $ g$ modulo $ f$ since by \eqref{eq:lagrange} and \eqref{eq:interpolation}  we have
$$
 \sum_{i=1}^n    g(\xi_i)\, u_i^2 + 2\sum_{i=1}^{(n-k)/2} |g(\xi_{k+2i})|\,u_{k+2i-1} u_{k+2i}\equiv \sum_{i=1}^n    g(\xi_i)\, u_i \equiv g\mod f .  
$$

Inspired by this construction, we define 
for fixed $\lambda_i \in \R$, $1\le i\le \frac{n-k}{2}$, the polynomial 
\begin{equation}\label{eq:psos}
h = \sum_{i=1}^{n} g(\xi_i)\, u_i^2   + 2\sum_{i=1}^{(n-k)/2} \lambda_i\, u_{k+2i-1} u_{k+2i}
\end{equation} which is also congruent to $g$ modulo $f$ for any choice of $\lambda_i$, $1\le i\le \frac{n-k}{2}$.\\

Next proposition shows that for a range of values of $\lambda_i$, this polynomial $h$ is a sum of $n$ linearly independent squares. 
\begin{proposition} \label{prop:homega} Let $f\in \R[x]$ be a squarefree polynomial as in Notation~\ref{not:f} and let $g\in \R[x]$ be such that $g(\xi_i)>0$ for $1\le i\le k$.\\
Fix $\lambda_i>|g(\xi_{k+2i})|$, $1\le i\le \frac{n-k}{2}$, and let  
$$ h = \sum_{i=1}^{n} g(\xi_i)\, u_i^2   + 2\sum_{i=1}^{(n-k)/2} \lambda_i\, u_{k+2i-1} u_{k+2i}$$ be the polynomial $h$ defined in \eqref{eq:psos}, which is congruent to $g$ modulo $f$. Then $h$  is a positive weighted sum of $n$ squares of linearly independent real polynomials of degree strictly bounded by $n$. More precisely, 
\begin{eqnarray}\label{eq:nsos}
 h &=& \sum_{i=1}^{n} \omega_i\, h_i^2
\end{eqnarray}
where
\begin{itemize}
    \item $h_i=u_i$ and  $\omega_i= g(\xi_i)$ \quad for  $1\le i\le k$,\\
    
    \item $h_{k+2i-1}= \Re(u_{k+2i}) - \dfrac{\Im(g(\xi_{k+2i}))}{ \lambda_i +\Re(g(\xi_{k+2i}))}\, \Im(u_{k+2i})$,  \\
  
  \noindent  
 $h_{k+2i}= \dfrac{\sqrt{ \lambda_i^2-|g(\xi_{k+2i})|^2}}{\lambda_i + \Re(g(\xi_{k+2i}))} \Im(u_{k+2i})$,\\
 
 \noindent 
  and   $\omega_{k+2i-1}=\omega_{k+2i} ={2\,(\lambda_i + \Re(g(\xi_{k+2i})))}$  \quad for $1\le i\le \frac{n-k}{2}$.
  
  \noindent (Here $\Re$ and $\Im$ denote real and imaginary part respectively.)
\end{itemize}

\end{proposition}
\begin{proof}
We first show that  the expressions in \eqref{eq:psos} and \eqref{eq:nsos} coincide.

Set $\gamma_i:=g(\xi_{k+2i})$ for $1\le i\le \frac{n-k}{2}$.
Applying the identity 
\begin{eqnarray*}
\lefteqn{ (a + \ib b) (u+ \ib v)^2
+ (a - \ib b) (u - \ib v)^2 + 2\,\lambda |u+\ib v|^2 }\\
&=& 2\,\big( (\lambda+a)\,u^2 - 2\,b\, u\,v + (\lambda-a)\, v^2 \big)\\
&=& {2 (\lambda + a)}\,\big( ( u - {b \over  \lambda + a} \, v)^2 +  (\lambda^2-a^2 - b^2)\, ({v\over \lambda + a}) ^2\big)
\end{eqnarray*}
for $\lambda + a\ne 0$, we get from Identity~\eqref{eq:psos}:
\begin{align*}
 h &= \ \sum_{i=1}^{k} g(\xi_i)\, u_i^2
   + \sum_{i=1}^{(n-k)/2}\Big( \gamma_i\, u_{k+2i}^{2}+ 
   \overline{\gamma}_i\,\overline{u_{k+2i}}^{2}+ 2\, \lambda_i \,|u_{k+2i}|^2\Big)\\
   &= \  \sum_{i=1}^{k} g(\xi_i)\, u_i^2\\
   &\quad + \sum_{i=1}^{(n-k)/2} 2\,(\lambda_i + \Re(\gamma_i))\, 
   \big(\Re(u_{k+2i}) -{\Im(\gamma_i)\over \lambda_i + \Re(\gamma_i)}\, 
   \Im(u_{k+2i})\big)^2\\
   &\quad + \sum_{i=1}^{(n-k)/2} 2\,(\lambda_i + \Re(\gamma_i)) \big({\sqrt{ \lambda_i^2-|\gamma_i|^2}\over\lambda_i + \Re(\gamma_i)} \Im(u_{k+2i})\big)^{2}
\end{align*}
since  $\lambda_i>|\gamma_i|$ implies $\lambda_i + \Re(\gamma_i) \ne 0$ and $\lambda_i^2-|\gamma_i|^2>0$.

Now,  observe that $\omega_i>0$ since for $1\le i\le k$, $\omega_i:=g(\xi_i)>0$ by assumption, and for $1\le i\le \frac{n-k}{2}$,
$\omega_{k+2i-1}=\omega_{k+2i}:=2(\lambda_i + \Re(\gamma_i)) >0$.
 Therefore $h$ is a positive weighted sum of $n$ squares of   polynomials of degree $<n$ with real coefficients. 

Finally, as the polynomials $u_i$ are linearly independent over $\C$ and $u_{k+2i-1}=\overline{u_{k+2i}}$, the real polynomials
$$u_1, \ldots, u_k, \Re(u_{k+2}), \Im(u_{k+2}), \ldots, \Re(u_{n}), \Im(u_{n})$$ are also linearly independent. 
This implies that  the real polynomials 
$h_1, \ldots, h_n$ are  also linearly independent over $\R$ (and in particular non-zero).
\end{proof}

\medskip

We fix the following notation for the rest of the paper:
\begin{notation}
We set $S^n(\R)$ for the set of symmetric matrices in $\R^{n\times n}$ and 
 $S_+^n(\R)$ for its cone of symmetric positive semidefinite matrices. 
 We equip $S^n(\R)$ with the Frobenius  inner product  $\sprod{A,B}= \mathrm{trace}(A \,B)$, $\forall A,B \in S^{n}(\R)$, which induces the Frobenius norm $\|\cdot\|$. The $2$-norm on the coefficients of polynomials in $ \R[x]$ is also denoted by $\|\cdot\|$.
For $m\in \N_0$ we set  $\R[x]_{m}$ for the set of polynomials of degree bounded by $m$.  
Finally,  $ \xx=[1, x, \ldots, x^{n-1}]^T$ is the column vector of monomials of degree $<n$.  
\end{notation}

Note that for any polynomial $p=\sum_{i=0}^d p_ix^i$, one has
\begin{align}\label{eq:boundpol}\|p\,f\|&\le \, \sum_{i=0}^d \|p_ix^i f\| \,\le \, \sum_{i=0}^d |p_i|\,\|x^i f\|\\& \le \, \sum_{i=0}^d \|p\|\,  \|f\|\,\le \, (d+1)\|p\|\, \|f\|\nonumber.\end{align}
\\

As a first corollary of Proposition~\ref{prop:homega}, we have:
\begin{corollary}\label{cor:L} Let $f,g\in \R[x]$ with $f$ of degree $n$ with simple roots $\xi_i,\ 1\le i\le n$, and $g$ of degree $<n$ that is  strictly positive at the real roots $\xi_1,\dots,\xi_k$ of $f$. Then, there exists a pair $(Q,q)\in S^n(\R)\times \R[x]_{n-2}$  with $Q$  {\em positive definite} such that $g=\xx^T Q \xx + q \,f$. \\ In particular  $Q\in \Int(S_+^n(\R))$, where $\Int$ denotes  interior. 
\end{corollary}

\begin{proof} For fixed $\lambda_i>|g(\xi_i)|$, $1\le i\le {n-k\over 2}$, let $H$  be the coefficient matrix of the polynomials $h_1, \ldots, h_n$ of Proposition~\ref{prop:homega} in the monomial basis $1, x, \ldots, x^{n-1}$, so that
\begin{equation}\label{eq:H}
[h_1,\dots,h_n]=\xx^T\, H.
\end{equation}
The matrix $H$ is invertible since $h_1, \ldots, h_n$ are linearly independent.
Let $\Delta$ be the diagonal matrix 
\begin{equation}\label{eq:Delta}
\Delta=\diag(\omega_1, \ldots, \omega_n).
\end{equation}
 
Then \eqref{eq:nsos} rewrites as 
\begin{equation}\label{eq:decomp}
   h = \xx^T H \Delta H^T \xx = \xx^T {Q}\, \xx 
\end{equation}
where ${Q} := H \Delta H^T $ is  positive definite since $H$ is invertible  and $\omega_i >0$ for $1\le i\le  n$. Also,
as $h \equiv g \mod f$ and $\deg(h)\le 2n-2$, there exists ${q}\in \R[x]_{n-2}$  such that $g = h + {q}\, f$. \\
Finally  ${Q}\in \Int(S_+^n(\R))$ since $\det({Q})>0$.
\end{proof}

\begin{remark}\label{rem:LDL}
The passage from $h=\sum_{i=1}^n\omega_ih_i^2$ with $\omega_i\in \R_{>0}$, $1\le i\le n$, to $h=\xx^TQ\xx$ where $Q\in S_+^n(\R)$ is a positive definite matrix, and vice-versa, is quite standard: 

As shown in the proof of Corollary \ref{cor:L}, 
$Q=H\Delta H^T$ where $H$ and $\Delta$ are defined in \eqref{eq:H} and \eqref{eq:Delta}.

Conversely, an exact  square-root-free Cholesky decomposition  of a positive definite matrix $Q\in S_+^n(\R)$ yields
$$
Q= LDL^T,
$$
where L is a lower unitriangular  matrix  and  $D$ is a diagonal matrix with  positive entries.
For instance, this decomposition can be computed exactly over $\Q$ through LU decomposition via Gaussian elimination of matrix Q. 

Then $\omega_1,\dots,\omega_n$ are the diagonal entries of $D$ and 
$$
[h_1, \ldots, h_n]:= {\bf x}^T L.
$$
\rightline{$\square$}
\end{remark}

Note that when $\lambda_i=|g(\xi_{k+2i})|$, which is the case in  Parrilo's polynomial~\eqref{eq:parrilo pol}, 
$h_{k+2i}=0$ and therefore these polynomials $h_i$, $1\le i\le n$, are not linearly independent. This means  that Parrilo's polynomial~\eqref{eq:parrilo pol} lies in the border of the cone  $S^n_+(\R)$. What we were able to do in Proposition~\ref{prop:homega} is to modify Parrilo's construction in order to obtain a polynomial $h$ in the interior of this cone. This   gives room to perturb it a little in order to get a rational polynomial with the same characteristics, and  yields the particular version of our main theorem when $g$ is strictly positive at all the real roots of a squarefree polynomial  $f$.
To describe this construction, we introduce the following ingredients.
\begin{notation}
\label{def:pi_p}
Let  $p=p_0 + p_1 x + \cdots + p_{2n-2} x^{2n-2}\in \R[x]$. We define the affine space 
$$
\cQ_p=\{ Q\in S^n(\R) \ : \ \xx^T Q \,\xx = p\},
$$ and  the
symmetric matrix
\begin{equation}\label{eq:Qp}
Q_p=\left[ \begin{array}{ccccc}p_0&\frac{p_1}{2}& \dots &\frac{p_{n-2}}{n-1}& \frac{p_{n-1}}{n}\\ 
\frac{p_1}{2}& &\addots & \addots & \frac{p_{n}}{n-1}\\
\vdots & \addots & \addots & \addots & \vdots\\
\frac{p_{n-2}}{n+1}&\addots  &  \addots& & \frac{p_{2n-3}}{2}\\ 
\frac{p_{n-1}}{n}& \frac{p_n}{n-1}& \dots &  \frac{p_{2n-3}}{2}&p_{2n-2}
\end{array}\right],
\end{equation}
which  satisfies
\begin{equation}\label{eq:xQpx}
\xx^T Q_p \xx = \sum_{k=0}^{2n-2} p_k x^k = p,
\end{equation} and therefore $Q_p\in \cQ_p$. 
\end{notation}

We note for further use that we have 
\begin{equation}\label{eq:bound}
\|Q_p\|=
\big(\sum_{k=0}^{2n-2} s_{k} \Big(\frac{p_k}{s_{k}}\Big)^2 \big)^{1/2} \le \big(\sum_{k=0}^{2n-2}p_k^2\big)^{1/2} =\|p\|,
\end{equation}
where 
\begin{equation}\label{eq:sk}s_k=s_{2n-2-k}=k+1,\quad 0\le k\le n-1,\end{equation} denotes the number of entries  in each of the $2n-1$ antidiagonals of $Q_p$.

\smallskip

We now describe the orthogonal projection from $S^n(\R)$  on $\cQ_p$  for the Frobenius norm, in a more convenient matrix formulation for the univariate case than in \cite[Prop.7]{PeyrlParrilo}, and prove it for sake of completeness. 
\begin{lemma}\label{lem:proj}
The  map
\begin{equation*}
    \pi_p:  S^n(\R) \longrightarrow  \cQ_p  \quad , \quad Q \longmapsto Q - Q_{\xx^T Q \xx-p} 
\end{equation*}
is the orthogonal projection onto the affine space $\cQ_p$ for the norm $\|\cdot \|$.
\end{lemma}

\begin{proof}
Let $Q\in S^n(\R)$.
By \eqref{eq:xQpx}, we have:
$$
\xx^T \pi_p(Q)\, \xx = \xx^T Q \xx- (\xx^T Q \xx -p)= p,
$$ and thus $\pi_p(Q)\in \cQ_p$.\\

To prove that $\pi_p(Q)$ is the orthogonal projection of $Q$ on $\cQ_p$, we show that $Q-\pi_p(Q)$ is orthogonal to $\cQ_p$:\\ We first observe that for any $Q\in S^n(\R)$,
\begin{equation*}\label{eq:xQx}
   \xx^T Q \xx = \sum_{k=0}^{2 n-2} 
\Big(\sum_{i+j=k+2} Q_{i,j}\Big) x^k =
\sum_{k=0}^{2 n-2} \sprod{Q,H_k} x^k,
\end{equation*}
where  for $0\le k\le 2n-2$, $H_k\in S^n(\R)$  is the Hankel matrix such that $(H_k)_{i,j}= 1$ if $i+j =k+2$ and $0$ otherwise, $1\le i,j\le n$.
This shows that 
the affine space $\cQ_p$ is  defined by the equations 
\begin{equation*}\label{eq:equationQp}
\sprod{Q,H_k} - p_k =0, \quad k=0,\ldots, 2n-2,
\end{equation*}
which implies that the vector space $\cQ_p^\perp$  orthogonal to $\cQ_p$ is spanned by $(H_k)_{0\le k\le 2n-2}$.\\
On another hand we can easily verify from its definition that 
\begin{equation*}\label{eq:QpwithHk}
Q_p = \sum_{k=0}^{2n-2}  \frac{p_k}{s_{k}} H_k,
\end{equation*}
where   $s_{k}$ is defined in \eqref{eq:sk}.  Therefore,
$$
Q - \pi_p(Q) = Q_{ \xx^T Q \xx-p} = \sum_{k=0}^{2n-2} \big(\sprod{Q,H_k}-\frac{p_k}{s_k} \big)  H_k,
$$
which shows that  $Q-\pi_p(Q)$ is a linear combination of  $(H_k)_{0\le k\le 2n-2}$, and thus orthogonal to  $\cQ_p$.
\end{proof}

We are going to use this projection to compute a rational sum of squares modulo $f$ for a polynomial $g$ strictly positive at the real roots of $f$. 
\begin{proposition}\label{prop:existence}
Let $f\in \Q[x]$ be a non-zero squarefree  polynomial  and $g\in \Q[x]$ be such that
$g$ is strictly positive at all the real roots of $f$. Then there exist   polynomials ${h}_i\in \Q[x]$ of degree $< n$ and  positive weights  ${\omega}_i\in \Q_{+}$, $1\le i \le n$, such that 
$$
h:=\sum_{i=1}^n {\omega}_i \,{h}_i^2 \quad\mbox{satisfies} \quad h \equiv g \mod f.
$$ 
\end{proposition}
\begin{proof}

There is a natural proof of this proposition which makes use of the fact that the set 
$$\{(A,b)\in S^n(\R)\times \R[x]_{n-2}: \, g=\xx^TA\xx+b\,f \} $$ is a real
affine space which in the case that   $f,g\in \Q[x]$ is defined by a rational basis and a rational particular point. This approach  
follows   the proof of the analogous result for the global case mentioned as  {\em image representation} in \cite[Section 3.2]{PeyrlParrilo}.

Here, we give the proof that  uses the orthogonal projection $\pi_p$ defined in Definition \ref{def:pi_p}, as done for the global case in the  {\em kernel representation}  in \cite[Section 3.1]{PeyrlParrilo}.

\smallskip

Without loss of generality we can assume that $\deg(g)<n$ by replacing it by its remainder modulo $f$.

Let $(Q^*, q^*)$ be given by  Corollary~\ref{cor:L}, i.e. $g=\xx^TQ^*\,\xx+q^*f$ and $Q^*\in \Int(S_+^n(\R))$,  and let $\sigma>0$ be the smallest eigenvalue of $Q^*$, which is the distance of $Q^*$ to the set of singular matrices, so that the open ball centered at $Q^*$ and of radius $\sigma$ is contained in $S_+^n(\R)$.

Take a rational approximation $(\overline Q,  q)\in S^n(\Q)\times \Q[x]_{n-2}$ such that 
\begin{align}\label{eq:approxbound}
\|\overline Q- Q^*\|< \frac{\sigma}{2} \quad \mbox{and} \quad \| q-q^*\|<\dfrac{\sigma}{2(n-1)\|f\|}.  \end{align}

The problem is that most surely, $\xx^T \overline Q \xx +   q\,f \ne g $.

Let  $e:={\bf x}^T {\overline Q} {\bf x} +qf-g$ be the error polynomial, and define 
$$
Q:=\pi_{g-qf}(\overline Q)= \overline{Q}-Q_e\in S^n(\Q),
$$ which is   the orthogonal projection of $\overline Q$ on  $\cQ_{g-qf}$ according to Lemma~\ref{lem:proj}. 
Then $Q\in \cQ_{g-qf}$, i.e. ${\bf x}^TQ{\bf x}+qf=g$. 

\smallskip
Next we  prove that $Q\in \Int(S_+(\Q))$ by proving that $\|Q-Q^*\|<\sigma$. We have 
\begin{align*}
 \| Q-Q^*\| &\le  \|\pi_{g-q\, f}(\overline Q)-\pi_{g-q\, f}({Q^*})\| + \|\pi_{g-q\, f}(Q^*)-Q^*\|\\
 &\le  \|\pi_{g-q\, f}(\overline Q)-\pi_{g-q\, f}({Q^*})\| + \|\pi_{g-q\, f}(Q^*)-\pi_{g-q^*\, f}(Q^*)\|\\
 &\le  \|\overline Q-{Q^*}\| + \|Q^*-Q_{\xx^T Q^* \xx-(g-q\, f)}-\big(Q^*-Q_{\xx^T Q^*\xx-(g-q^*\, f)}\|\\
 & \le  \|\overline Q-{Q^*}\| + \|Q_{(q^*-q)\, f}\|  ,
\end{align*}
since $\xx^TQ^*\xx+q^*g=g$ implies $Q^*=\pi_{g-q^*f}(Q^*)$.\\
By \eqref{eq:bound} and \eqref{eq:boundpol} we have 
\begin{align*} \|Q_{(q^*-q)\, f}\| \le  \|(q^*-q)\, f\|  \le (n-1) \|q^*-q\| \, \|f\| \end{align*}
since $\deg(q^*-q)\le 2n-2$.
Finally, by \eqref{eq:approxbound}, we conclude
$$\|Q- Q^*\| < \frac{\sigma}{2} + (n-1) \frac{\sigma}{2(n-1)\|f\|} \, \|f\|= \sigma.
$$
This implies that 
$ Q\in \Int(S_+^n(\R))$, i.e. $ h=\xx^T Q \,\xx$ is a rational positive weighted sum of squares.
\end{proof}

\begin{example} \label{ex:toy1} We now consider a toy example to illustrate our construction. This is a toy example because in this case we know the roots of $f$ and use that knowledge, as in the proof or our existential theorem. 

\smallskip \noindent 
Let $f=x^3-2=(x-{2^{1/3}})(x-2^{1/3}\omega)(x-2^{1/3}\overline\omega)$, where $\omega=e^{2\pi \ib/3}$,  and $g=x$, which is strictly positive at $2^{1/3}$. Set $\xi_1=2^{1/3}$, $\xi_2=2^{1/3}\omega$ and $\xi_3=\overline{\xi_2}$.

Parrilo's construction \eqref{eq:parrilo pol} gives in this case the following real polynomial, which is congruent to $g$ modulo $f$ and a sum of 2 squares:
$$g(\xi_1){\underbrace{u_1(x)}_{\in\R[x]}}^2+{\underbrace{\left(\sqrt{g(\xi_2)} u_2(x)+\sqrt{g(\xi_3)}u_3(x)\right)}_{\in \R[x]}}^2\ = \  \xx^T Q^* \,\xx
$$
where $$Q^*=    {\left[\begin{matrix} \frac{2\sqrt[3]{2}}{9}&\frac{2}{9}&-\frac{\sqrt[3]{4}}{18}\\[1mm]
\frac{2}{9}&\frac{\sqrt[3]{4}}{18} &-\frac{\sqrt[3]{2}}{18}\\[1mm]
-\frac{\sqrt[3]{4}}{18}&-\frac{\sqrt[3]{2}}{18} &\frac{5}{18}\\
\end{matrix}\right]} .$$
Note that $Q^*$ is a rank 2 positive semidefinite   matrix, which therefore lies in the border of the cone of positive semidefinite matrices. 

Now, if we take $\lambda:=2|g(\xi_2)|= 2\cdot 2^{1/3}$ in our construction \eqref{eq:psos}, we get 
$h^*=  \xx^TQ^*\, \xx$ where 
$$Q^*=\left[\begin{matrix} \frac{4\sqrt[3]{2}}{9}&\frac{1}{9}&-\frac{\sqrt[3]{4}}{9}\\[1mm]
\frac{1}{9}&\frac{2\sqrt[3]{4}}{9} &-\frac{\sqrt[3]{2}}{9}\\[1mm]
-\frac{\sqrt[3]{4
}}{9}&-\frac{\sqrt[3]{2}}{9} &\frac{7}{18}\\
\end{matrix}\right]$$ 
is a (rank 3) definite positive matrix with smallest eigenvalue $\sigma\sim 0.2239$, and 
$$g=\xx^T Q^*\,\xx + q^* \,f \quad \mbox{for} \quad q^*= -\frac{7}{18}x+\frac{2\sqrt[3]{2}}{9}.$$
Here, if we take the following rational approximations of $Q^*$ and $q^*$ (rounding to two significant digits)
$$\overline Q= \left[\begin{matrix}0.6& 0.1& -0.2\\ 0.1& 0.4& -0.1\\-0.2& -0.1& 0.4
\end{matrix}\right] \quad \mbox{ and } \quad  q=-0.4x+0.3 $$
 we get that
$\|Q^*-\overline Q\|\cong 0.0923$  and $\|q^*-{q}\|\cong 0.0229$. Thus, we have  $\|Q^*-\overline Q\|<\frac{\sigma}{2}\cong 0.112$ and $\|q^*-{q}\|<\frac{\sigma}{2(n-1)\|f\|}\cong0.025$ respectively, satisfying both of the  bounds given in \eqref{eq:approxbound}  required in the proof of Proposition~\ref{prop:existence}.  
We have $\xx^T \overline Q\,\xx +  q\,f= 0.1x^3+x \ne g$, with error 
 $$e=\xx^T \overline Q\,\xx+ q\, f-g=0.1x^3.$$
 Compute the orthogonal projection of $\overline Q$ on $ \cQ_{g-qf}$:
$$ Q =\pi_{g-qf}(\overline Q)=\overline Q - Q_e= \left[\begin{matrix}0.6& 0.1& -0.2\\ 0.1& 0.4& -0.15\\-0.2& -0.15& 0.4
\end{matrix}\right]$$ so that $\xx^T \overline Q \,\xx-\xx^T Q \,\xx= e$ and $ Q$ is still a definite positive matrix.
Then matrix $ Q\in S^3(\Q)$ satisfies 
$$ h:=\xx^T  Q \,\xx = \xx^T  \overline Q \,\xx-e= g- q\, f \equiv g \mod f,$$
and $  h$ is a sum of squares of rational polynomials, which we can obtain applying the square-root-free Cholesky decomposition of $Q$ (Remark~\ref{rem:LDL}) as follows:
$$
h=\frac{3}{5}\left(1 + \frac{1}{6}x - \frac{1}{3}x^2\right)^2+ \frac{23}{60}\left( x - \frac{7}{23}x^2\right)^2+\frac{137}{460} x^4.
$$
\end{example}

\subsection{The general case}\label{sub:2}

In this subsection we generalize the results of the previous section to the case when $f$ is non-necessarily squarefree and  $g$ is non-negative at  all the real roots of $f$ (but might vanish on some of them), as long as $\gcd(f,g)$ and   $f/\gcd(f,g)$   are relatively prime, in order  to obtain our main theorem.

We will need the following auxiliary results, namely Hensel lemma and Chinese remainder theorem.

\begin{lemma}
 \label{lem:Hensel} Let $p,g\in \Q[x]$ with $p$  irreducible in $\Q[x]$ which does not divide $g$. Assume that there exists $\overline{h}_1, \ldots, \overline{h}_N\in \Q[x]$ and $\omega_1\ldots, \omega_N\in \Q_+$ for some $N\in \N$  with $\deg(\overline{h}_i)<\deg(p)$ such that 
 $$
 g\equiv \sum_{i=1}^N\omega_i \overline{h}_i^2 \mod p.
 $$
 Then for any fixed  $e\in \N$, $e\geq 1$, there exist $h_1, \ldots, h_N\in \Q[x]$ with $\deg(h_i)<e\cdot\deg(p)$ such that 
 $$
 g\equiv \sum_{i=1}^N \omega_ih_i^2 \mod p^{e}.
 $$
\end{lemma}

\begin{proof}
We show that it suffices to perform Hensel lifting on one of the polynomials $\overline{h}_i$.
Since $p\in \Q[x]$ is irreducible and does not divide $g$, one of the $\overline{h}_i$ at least is not divisible by  $p$, and w.l.o.g. we assume that it is $\overline{h}_1$. 

Define 
$$ \overline g = \frac{g}{\omega_1} \ \mbox{ for } \ N=1 \quad \mbox{and} \quad 
\overline{g}:=\frac{g- \sum_{i=2}^N \omega_i\overline{h}_i^2}{\omega_1} \in \Q[x] \ \mbox{ for } \ N>1.
$$
Then $
\overline{g}\equiv \overline{h}_1^2 \mod p
$, and 
we define the following Newton iteration starting from $h_1^{(0)}:=\overline{h}_1$:
\begin{equation}\label{def:newton} h_1^{(k+1)}\equiv \frac{1}{2}\Big( h_1^{(k)} + \frac{\overline g}{h_1^{(k)}} \Big) \equiv \frac{1}{2}\Big( h_1^{(k)} + s_1^{(k)}\,\overline g \Big) \mod p^{2^{k+1}} \quad  \mbox{for } k\ge 0,\end{equation}
where $s_1^{(k)}\in \Q[x]$ is defined by $s_1^{(k)}\, h_1^{(k)}\equiv 1 \mod p^{2^{k+1}}$.

First note that this sequence is well defined in $\Q[x]$ since by induction, $$h_1^{(k)} \equiv \frac{1}{2}\Big( h_1^{(0)} + \frac{\big(h_1^{(0)}\big)^2}{h_1^{(0)}} \Big)\  \equiv \  h_1^{(0)} \mod p ,$$  and therefore $h_1^{(k)}$ is prime to the irreducible polynomial $p$ since $h_1^{(0)}$ is, and hence invertible modulo $p^{2^{k+1}}$.

We now prove by induction that $\big(h_1^{(k)})^2\equiv \overline g \mod p^{2^k}$:

First, from \eqref{def:newton} we derive 
\begin{equation}\label{id:newton2}h_1^{(k)}\,h_1^{(k+1)}\equiv \frac{1}{2}\Big( \big(h_1^{(k)})^2 + \overline g \Big) \mod p^{2^{k+1}} .\end{equation}

Now,
by the inductive hypothesis, $\big(h_1^{(k)}\big)^2  \equiv \overline g \mod p^{2^k}$  implies that 
    
\begin{align*}  h_1^{(k+1)}&\equiv \frac{1}{2}\Big( h_1^{(k)} + s_1^{(k)}\,\big(h_1^{(k)}\big)^2 \Big) \mod p^{2^{k}}\\
& \equiv \frac{1}{2}\Big( h_1^{(k)} + h_1^{(k)}  \Big)\ \equiv \ h_1^{(k)} \mod p^{2^{k}}.\end{align*}
Therefore,
$h_1^{(k)} = h_1^{(k+1)} + t\,p^{2^k}$ for some $t\in \Q[x]$ and from \eqref{id:newton2},
\begin{align*}\big(h_1^{(k+1)}+ t\,p^{2^k}\big)h_1^{(k+1)}&\equiv \frac{1}{2}\Big( \big(h_1^{(k+1)}+t\,p^{2^k}\big)^2 + \overline g \Big) \mod p^{2^{k+1}} \\ 
&\equiv \frac{1}{2}\Big( \big(h_1^{(k+1)}\big)^2+2t\,p^{2^k}h_1^{(k+1)} + \overline g \Big)  \mod p^{2^{k+1}} \end{align*}
and we can cancel $t\,p^{2^k}h_1^{(k+1)}$ from both sides. We conclude
$$\big(h_1^{(k+1)})^2\equiv \frac{1}{2}\Big( \big(h_1^{(k+1)}\big)^2 + \overline g \Big) \mod p^{2^{k+1}}$$
which implies $$\big(h_1^{(k+1)}\big)^2\equiv \overline g \mod p^{2^{k+1}}.$$
Going back to the definition of $\overline g$, 
$$ g=  \omega_1\overline g +  \sum_{i=2}^N \omega_i\overline{h}_i^2 \equiv
\omega_1\big(h_1^{(k)}\big)^2+ \sum_{i=2}^N \omega_i\overline{h}_i^2 \ \mod p^{2^k}. $$
Finally, if we choose $k$ such that $2^{k-1}<e\leq 2^k$ and define $h_1:=h_1^{(k)} \mod p^e$, $h_i:=\overline{h}_i$  for $2\le i\le  N$ and $\omega_1,\ldots,  \omega_N$ unchanged then  we get the sum of squares decomposition of $g$ modulo $p^e$ with the desired  degree bounds. 
\end{proof}

\begin{lemma} \label{lem:CRT}
Let $f_1, \ldots, f_r\in \Q[x]$ with $\gcd(f_i,f_j)=1$ for  $1\leq i<j\leq r$.
Assume that $g\in \Q[x]$ satisfies 
$$
g\equiv \sum_{j=1}^{N_i} \omega_{i,j}h_{i,j}^2 \mod f_i, \quad  1\le i\le r,
 $$
for some $N_i\in \N$, $h_{i,j}\in \Q[x]$ with $\deg(h_{i,j})<\deg f_i$ and $\omega_{i,j}\in \Q_+$,  for $1\le j\le  N_i$.  Then there exist $N\in \N$, $h_1, \ldots, h_N\in \Q[x]$ and $\omega_1, \ldots, \omega_N\in \Q_+$ such that 
$$
g\equiv \sum_{i=1}^{N} \omega_ih_{i}^2 \mod \Big(\prod_{i=1}^r f_i\Big).
$$
Furthermore,  $\deg(h_i)<\sum_{i=1}^r \deg(f_i)$ for $1\le i\le  N$.
\end{lemma}

\begin{proof} The usual Chinese remainder theorem for a system 
$$g\equiv g_i \mod f_i, \ 1\le i\le r,$$ 
admits the solution (c.f. \cite[Algorithm 5.4]{MCA})
$$g\equiv s_1f^{(1)}g_1+\cdots + s_r f^{(r)} g_r\mod f$$  
where 
$f:=\prod_{i=1}^r f_i$,  $f^{(i)}:=\prod_{j\ne i}f_j $ 
and $s_i$ is defined by $s_if^{(i)}+t_if_i=1$  for $1\le i\le r$.

On another side, notice that $$s_i f^{(i)} \equiv (s_if^{(i)})^2 \mod f, \ 1\le i\le r,$$ since $s_if^{(i)}\equiv 1 \mod f_i$ and $s_if^{(i)}\equiv 0 \mod f_j$ for $j\ne i$. Then 
$$g\equiv (s_1f^{(1)})^2g_1+\cdots + (s_rf^{(r)})^2g_r \mod f.$$
In our setting, since 
$g_i:= \sum_{j=1}^{N_i} \omega_{i,j}h_{i,j}^2$ we get
\begin{align*}g&\equiv \sum_{i=1}^r (s_if^{(i)})^2\Big(\sum_{j=1}^{N_i} \omega_{i,j}h_{i,j}^2\big) \mod f\\
&\equiv \sum_{i=1}^r \sum_{j=1}^{N_i} \omega_{i,j}( s_if^{(i)}h_{i,j})^2 \mod f.
\end{align*}
We get  $N:=\sum_{i=1}^r N_i$ and reduce $s_if^{(i)}h_{i,j}$ modulo $f$ to achieve the desired degree bounds in the SOS decomposition.  
\end{proof}

We are now able to prove the full version of our theorem. We repeat the statement here for the reader's convenience.

\smallskip
\noindent{\bf Theorem.} {\em  Let $f\in \Q[x]$ be a non-zero   polynomial  of degree $n$ and $g\in \Q[x]$ be such that
$\gcd(f,g)$ and $f/\gcd(f,g)$ are relatively prime. Assume that 
$g$ is non-negative at all the real roots of $f$. Then there exist   polynomials ${h}_i\in \Q[x]$ of degree $< n$ and  positive weights  ${\omega}_i\in \Q_{+}$, $1\le i \le N$ for some $N\in \N$, such that 
$$
h:=\sum_{i=1}^N{\omega}_i \,{h}_i^2 \quad\mbox{satisfies} \quad h \equiv g \mod f.
$$}

\smallskip

\begin{proof}
First assume that $\gcd(f,g)=1$. Note that therefore, the assumption that $g$ is non-negative at the real roots of $f$ implies that  $g$ is strictly positive at the real roots of $f$. 

W.l.o.g. we can assume that $f$ is monic. Suppose $f$ has the following decomposition over $\Q$ into powers of irreducible factors in $\Q[x]$
$$
f=p_1^{e_1}\cdots p_r^{e_r}
$$ 
where $p_i$ are distinct monic irreducible polynomials in $\Q[x]$,  $e_i\in \N$ for $1\le i\le  r $, and $\sum_{i=1}^r e_i\deg(p_i)=n$. 

Fix $i\in \{1, \ldots, r\}$. Since $g$ is strictly positive at the real roots of the irreducible polynomial $p_i$, we can apply  Proposition \ref{prop:existence} to  $p_i$ and $g$, which shows  the existence of   $\overline{h}_{i,j}\in \Q[x]$ of degree $< N_i:=\deg(p_i)$ and   ${\omega}_{i,j}\in \Q_{+}$ for $1\le i \le N_i$, such that 
$$
g\equiv   \sum_{j=1}^{N_i} {\omega}_{i,j} \,\overline{h}_{i,j}^2 \mod p_i.
$$
Next we apply  Lemma \ref{lem:Hensel} with $p=p_i$ and $e=e_i$ to show the existence of    ${h}_{i,j}\in \Q[x]$ of degree $< e_i\deg(p_i)$, $1\le i \le N_i$, such that 
\begin{equation}\label{eqn:qiei}
 g \equiv \sum_{j=1}^{N_i} {\omega}_{i,j} \,{h}_{i,j}^2  \mod p_i^{e_i} .
\end{equation}
Finally we apply Lemma \ref{lem:CRT} with $f_i=p_i^{e_i}$ for $1\le i\le r$ to combine the congruences in \eqref{eqn:qiei} and obtain $N\in \N$, $h$, $h_1, \ldots, h_N\in \Q[x]$ and $\omega_1, \ldots, \omega_N\in \Q_+$ such that $$
h:=\sum_{i=1}^N{\omega}_i \,{h}_i^2 \quad\mbox{satisfies} \quad h \equiv g \mod f.
$$
Furthermore,  $\deg(h_i)<\sum_{i=1}^re_i\deg(p_i)=n$ for $1\le i\le  N$. This proves the claim for an arbitrary polynomial $f$ with  $\gcd(f,g)=1$.

\smallskip
Assume now that $d:=\gcd(f,g)\neq 1$. We  show that under our assumption $\gcd(f/d,d)=1$, there is a polynomial $b\in \Q[x]$ relatively prime to $f/d$ which satisfies that  $b\,d^2\equiv g \mod f$ and therefore $b$ is strictly positive at the real roots of $f/d$:

The assumption implies that $\gcd (f/d,g)=1$, and therefore $g$ is strictly positive at the real roots of $f/d$. Since  $\gcd(f/d,d^2)=1$ as well, there exist $s,t\in \Q[x]$ s.t. 
$$ 
1=s\cdot \frac{f}{d} + t\cdot d^2.
$$
This implies in particular $\gcd(f/d,t)=1$ and that 
\begin{equation}\label{eq:tg}g= s\cdot \frac{g}{d}\cdot f + (tg) \cdot d^2.\end{equation}
We set $b:=tg$. Then $b$ and $f/d$ are relatively prime since $t$ and $f/d$, and $g$ and $f/d$ are. 
Therefore $b$ is striclty positive at  the real roots of $f/d$ because for any such root $\xi$, $d(\xi)\ne 0$, $b(\xi)\ne 0$ and $b(\xi)d^2(\xi)=g(\xi)\ge 0$. 

Finally, \eqref{eq:tg} implies that $b\,d^2\equiv g \mod f$. 

We then apply our previous construction to $b$ and $f/d$: There exist     $\overline{h}_i\in \Q[x]$ of degree $< n-\deg(d)$ and   ${\omega}_i\in \Q_{+}$, $1\le i \le N$, such that 
$$\overline  h:=\sum_{i=1}^N {\omega}_i \overline {h}_i^2 \quad\mbox{satisfies} \quad  
\overline  h \equiv b \mod \frac{f}{d} .
$$
Therefore,
$$ d^2\,\overline h=\sum_{i=1}^N {\omega}_i \,(d\,\overline h_i)^2 \quad\mbox{and}\quad 
d^2 \overline  h \equiv b \,d^2 \mod f.$$
Since $b
\,d^2\equiv g \mod f$ we conclude that
$$
  d^2 \overline  h \equiv g  \mod f.
$$
We note that  $\deg(d\, \overline h_i)<n$, thus $h:=d^2 \overline h$,  $h_i:=d \, \overline h_i$ and $\omega_i$, $1\le i\le N$, satisfy the claim of the theorem.
\end{proof}

\begin{example}
Let us again consider a toy example to show how it works when $\gcd(f,g)\ne 1$ and $f$ is not squarefree.

Consider $f=x(x^3-2)^2$ and  $g=x^3$. Here  $d=\gcd(f,g)=x $ and  $f/d= (x^3-2)^2$ are relatively prime, so we are in the assumptions of our theorem.

In this case, as $g/d^2=x$ is already a polynomial, we can take $tg = g/d^2=x$.
\begin{enumerate}
\item Find rational SOS for $ g/d^2= x$ modulo $(x^3-2)$ (see Example \ref{ex:toy1}):
{\small$$
x\equiv \frac{3}{5}\left(1 + \frac{1}{6}x - \frac{1}{3}x^2\right)^2+ \frac{23}{60}\left( x - \frac{7}{23}x^2\right)^2+\frac{137}{460} x^4 \mod (x^3-2).
$$}
\item Apply Hensel lifting to find rational SOS for $ g/d^2=x$ modulo $(x^3-2)^2$ (note that we lift only the last term):
{\small
\begin{align*}
 x&\equiv \frac{3}{5}\left(1 + \frac{1}{6}x - \frac{1}{3}x^2\right)^2+ \frac{23}{60}\left( x - \frac{7}{23}x^2\right)^2\\
 &+\frac{137}{460} \left(\frac{-46}{137}x^5 + \frac{69}{274}x^4 + \frac{229}{137}x^2 - \frac{69}{137}x\right)^2 \mod (x^3-2)^2,
 \end{align*}}
 i.e. $  x\equiv \omega_1\overline h_1^2+\omega_2\overline h_2^2+\omega_3\overline h_3^2 \mod f/d=(x^3-2)^2$.

\item Multiply both sides by $d^2=x^2$:
$$ g\equiv \omega_1(x\overline h_1)^2+\omega_2(x\overline h_2)^2+\omega_3(x\overline h_3)^2 \mod f=x(x^3-2)^2.$$
\end{enumerate}
\end{example}

\section{The algorithm}\label{sec:algorithm}

In this section, we describe the  algorithm announced in the introduction, that computes a certificate of non-negativity of a polynomial $g\in\Q[x]$ at the real roots of another  polynomial $f\in\Q[x]$.

\subsection{Certificate for a strictly positive polynomial}
Here we assume that $f\in \Q[x]$ is a {\em squarefree} polynomial of degree $n$, and that
 $g$ is {\em strictly positive} at all the real roots of $f$. 
 
{We consider the following optimization problem
\begin{equation}\label{eq:pop}
\begin{array}{cl}
     \max  & \quad \lambda \\
     \mbox{s.t.} & Q - \lambda \, I \succcurlyeq 0\\
          & q\in \R[x]_{n-2}\\
          & g = \xx^T Q\, \xx + q\, f
\end{array}    
\end{equation}
where $\xx=[1,\ldots,x^{n-1}]^T$ is the vector of monomials of degree $< n$ and $Q\in S^n(\R)$. (Here $\succcurlyeq 0$ denotes positive semidefinite.)
It is finding the maximal $\lambda$, which is  bounded from above by all the eigenvalues of symmetric matrices $Q$ satisfying $g = \xx^T Q\, \xx + q\, f$. 
The set 
\begin{align*}
  \cC&=\{ (Q,q)\in S^{n}(\R) \times \R[x]_{n-2}  \ : \ Q\succcurlyeq 0,\ g = \xx^T Q \,\xx + q\, f\}
\end{align*} is convex, as the intersection of the linear space $$\{(Q,q)\in S^n(\R)\times \R[x]_{n-2},\,  g - \xx^T Q \,\xx - q\, f=0\}$$ with the convex cone $S^n_+(\R)\times \R[x]_{n-2}$. If the optimal value $\lambda^*$ of \eqref{eq:pop} is strictly positive, then 
the relative interior of $\cC$ is non-empty. 
}

{
By solving the convex optimization problem \eqref{eq:pop}
using a numerical interior point solver, working at a given precision $\mu$, we obtain an approximation of an interior point of $\cC$ where the objective function reaches its maximum $\lambda^*>0$. This yields a rational  approximation of an interior point $(Q^*,q^*)$ of the convex set $\cC$. That is, the numerical solver computes a (rational) approximate solution $(Q^*,q^*)$ of the optimization problem \eqref{eq:pop}, where if the precision $\mu$ is good enough and $\lambda^*>0$, $Q^*\in S_+^n(\Q)$ is a positive definite matrix  but    
there will be an error polynomial $\xx^TQ^*\xx+q^*f-g\ne 0$,  although close to 0.
}

Since $(Q^*,q^*)$ may have a lot of decimals, in order to obtain a rational decomposition of $g$ modulo $f$ of small size, we start by 
rounding,  at a convenient precision $\delta>0$, $Q^* \in S_+^n(\R)$ to a nearby $\overline{Q}\in S^n(\Q)$ and $q^*\in \R[x]_{\le n-2}$ to a nearby rational polynomial $q \in \Q[x]_{\le n-2}$. We  then compute the projection $Q:=\pi_{g-q\,f}(\overline Q)\in \cQ_{g-q\, f}$ which satisfies  $g = \xx^T Q \,\xx  + q\, f$. 
As in the proof of Proposition~\ref{prop:existence}, if  $\|Q-Q^*\|$ is smaller than the smallest eigenvalue $\sigma$ of $Q^*$, then 
$Q\in S_+^n(\Q)$ is a rational positive definite matrix and $g = \xx^T Q \xx  + q f$ gives a rational SOS decomposition   of $g$ modulo $f$, that is $(\xx^T Q \xx,q)$ is a rational certificate of positivity of $g$ at the real roots of $f$.

\smallskip
Given the approximate solution $(Q^*,q^*)$ output by the numerical solver, we detail in the following proposition a bound on the rounding precision $\delta$ chosen to define $(\overline Q,q)$ needed to guarantee that $Q=\pi_{g-qf}(\overline Q)$ is a positive definite matrix. We assume here that the matrix $Q^*$ output by the solver is positive definite.
\begin{proposition}
Let $\sigma>0$ be the smallest eigenvalue of $Q^*$ and assume that $\rho: = \|{ \xx^T Q^* \xx + q^*\, f - g}\|< \sigma$. Set $$0<\delta < \frac{1}{n + (n-1)\sqrt{n} \, \|f\| } (\sigma - \rho).$$
Then, for any rational approximations $(\overline Q,q)\in S^n(\Q)\times \Q[x]_{n-2}$ of $(Q^*,q^*)$  such that $$|\overline Q_{i,j}-Q_{i,j}^*|\le \delta, \ 1\le i,j\le n \quad \mbox{and} \quad |q_i-q_i^*|\le \delta, \ 0\le i\le n-2,$$ the symmetric matrix $Q =\pi_{g-q f}(\overline Q)\in S^n(\Q)$, which satisfies $g=\xx^TQ\,\xx+q\,f$, is  positive definite.
\end{proposition}
\begin{proof}
We have
$$\|\overline Q-Q^*\|\le n\, \delta  \quad \mbox{and} \quad \|q-q^*\|\le \sqrt{n}\,\delta.
$$
Then, the distance between $Q=\pi_{g-q\,f}(\overline Q)$ and $Q^*$ can be bounded, as in the proof of Proposition~\ref{prop:existence} but with the difference that $Q^*\ne \pi_{g-q^*f}(Q^*)$, as follows:
\begin{align*}
 \| Q-Q^*\| &\le  \|\pi_{g-q\, f}(\overline Q)-\pi_{g-q\, f}({Q^*})\| + \|\pi_{g-q\, f}(Q^*)-\pi_{g-q^*\, f}(Q^*)\|\\ & \quad + 
 \| \pi_{g-q^*\, f}(Q^*)-Q^*\|\\ & \le  \|\overline Q-{Q^*}\|\\ &  \quad +\|Q^*-Q_{\xx^TQ^*\xx-(g-qf)}- (Q^*-Q_{\xx^TQ^*\xx-(g-q^*f)} \| \\ & \quad +
 \|Q_{ \xx^T Q^* \xx - (g-q^*\, f)}\|\\ 
 &\le  \|\overline Q-{Q^*}\| + \|Q_{(q^*-q)f}\| + \|Q_{\xx^T Q^* \xx - (g-q^*\, f)}\|.
\end{align*}
Using \eqref{eq:bound} and \eqref{eq:boundpol}, as in the proof of Proposition~\ref{prop:existence} we have
$$
\|Q_{(q^*-q)\, f}\| \le  (n-1) \|q^*-q\| \, \|f\|\le (n-1) \sqrt{n} \,\delta \,\|f\|$$ and 
$$
\|Q_{\xx^T Q^* \xx - (g-q^*\, f)}\|  
\le \| \xx^T Q^* \xx + q^*\, f-g\| =\rho.
$$
As $\|\overline{Q}-Q^*\|\le n \,\delta$, we deduce that
\begin{align*}
\|Q- Q^*\| &\le  (n + (n-1)\sqrt{n} \, \|f\|)\, \delta + \rho < (\sigma-\rho)+\rho= \sigma 
\end{align*}
Therefore $Q$ is  positive definite.
\end{proof}

The approximation of $\sigma$ and the norm $\rho$ of the error polynomial $\xx^TQ^*\xx+q^*f-g$, which is approximately $0$, depend  on the precision $\mu$ of the solver. If $\rho> \sigma$, we need to increase the precision $\mu$ of the numerical solver and compute a new solution $(Q^*,q^*)$.

We can now summarize  the certification algorithm for a strictly positive polynomial, in Algorithm \ref{algo:strictpos}, which  is implemented in the function \texttt{exact\_decompose} of Julia package \texttt{MomentTools.jl}\footnote{\tt https://gitlab.inria.fr/AlgebraicGeometricModeling/MomentTools.jl}.
\begin{algorithm}\caption{\label{algo:strictpos}Rational SOS certificate modulo a squarefree polynomial for a strictly positive polynomial}
\textbf{Input:} $f \in \Q[x]_n$ squarefree, $g\in \Q[x]_{n-1}$ such that $g>0$ at the real roots of $f$.

\begin{enumerate}
  \item  $\mu \leftarrow \mu_0$ default precision of the interior point solver.
  \item $(Q^*,q^*) \leftarrow$ solution of the SDP problem \eqref{eq:pop} by the numerical interior point solver working at precision $\mu$;
  \item $\sigma \leftarrow$ smallest eigenvalue of $Q^*$;
  \item $\rho \leftarrow \|{\xx^T Q^* \xx + q^*\, f-g}\|$
  the 2-norm of the error polynomial; 
  \item $\delta \leftarrow \frac{0.99}{n + \sqrt{n} (n-1) \, \|f\| } (\sigma - \rho) $; 
     If $\delta < 0$ then increase precision $\mu\leftarrow 2 \, \mu$ and repeat from step $(1)$;
  \item $\overline{Q}\leftarrow$ round $Q^*$ to rational coefficients, with $\lceil \log_{10} (\delta^{-1})\rceil$ exact digits after decimal point;
  \item $q \leftarrow$ round $q^*$ to rational coefficients, with $\lceil \log_{10} (\delta^{-1})\rceil$ exact digits after decimal point;
  \item $Q\leftarrow \pi_{g-q\,f}(\overline{Q})$;
\end{enumerate}
\textbf{Output:} $(Q,q) \in S^n_{+}(\Q)\times \Q[x]_{n-2}$ such that 
\begin{itemize}
\item $g= \xx^T Q \xx + q\, f$,
\item $Q$ definite positive.
\end{itemize}
\end{algorithm}

\subsection{Certificate for a non-negative polynomial}
We consider now the case where $f$ arbitrary and $g$   non-negative at the real roots of $f$ satisfy the assumption that  $\gcd(f,g)$ and $f/\gcd(f,g)$ are relatively prime. We set $d \assign \gcd(f,g)$. 

We closely follow the proof of our main Theorem in Section~\ref{sec:exist}. We first compute $b\in \Q[x]$ relatively prime to $f/d$ such that $b$ is strictly positive at the real roots of $f/d$ and $b\,d^2\equiv g \mod f$.


We then compute the irreducible factorization of $f/d = \prod_{i=1}^r p_i^{e_ i}$ where the polynomials $p_i\in \Q[x]$ are irreducible, thus with simple roots, and pairwise relatively prime.

We observe that $b$ and $p_i$ are relatively prime, and 
that $b$ is strictly positive on the real roots of $p_i$,  $1\le i\le r$.

We set $b_i$ to be the remainder of $b$ modulo $p_i$, $1\le i\le r$, and we apply Algorithm \ref{algo:strictpos} to  $p_i$ and $b_i$. We get the rational SOS certificate
$$
 b_i = \xx^T Q_i\xx +  q_i\, p_i
$$
where, setting   $n_i:=\deg(p_i)$, $Q_i\in S^{n_i}_+(\Q)$ is positive definite and $q_i\in \Q[x]_{n_i-2}$. 
 We deduce from the square-root-free Cholesky factorisation of $Q_i$ (cf. Remark~\ref{rem:LDL})  an SOS decomposition
$$b_i \equiv \sum_{j=1}^{n_i}\omega_{i,j}\overline h_{i,j}^2 \mod p_i, $$
where $\omega_{i,j} \in \Q_+, \overline h_{i,j}\in \Q[x]$.

Therefore $$b\equiv \sum_{j=1}^{n_i} \omega_{i,j} \overline h_{i,j}^2 \mod p_i, \ \ 1\le i\le r.$$
By Hensel lifting (Lemma \ref{lem:Hensel}), we deduce an SOS decomposition of $b$ modulo $p_i^{e_i}$, and 
by the Chinese Remainder Theorem (Lemma \ref{lem:CRT}), we deduce an  SOS decomposition of $b$ modulo $f/d$:
$$
b \equiv  \sum_{i=1}^N \omega_{i} \overline{h}_{i}^2 \mod {f/d}
$$
with 
$\omega_{i} \in \Q_+, \overline{h}_{i}\in \Q[x]$.
Using that $b\,d^2\equiv g \mod f$, this gives the following SOS decomposition of $g$ modulo $f$:

$$
g \equiv  \sum_{i=1}^N \omega_{i} (d\overline{h}_{i})^2 \mod {f}
$$
and we finally compute  $q\in \Q[x]$ s.t. 
$$ g=\sum_{i=1}^N \omega_{i} (d\overline{h}_{i})^2 + q\,f.$$

This computation is summarized in Algorithm \ref{algo:nonneg}.

\begin{algorithm}\caption{\label{algo:nonneg}Rational SOS certificate for a non-negative polynomial}
\textbf{Input:} $f \in \Q[x]_n$, $g\in \Q[x]$ such that $g\ge 0$ at the real roots of $f$ and $\gcd(f,g)$ and $f/\gcd(f,g)$ are relatively prime.
\begin{enumerate}
  \item $d \leftarrow \makeatletter \gcd(f,g)$;
  \item Compute $b\in \Q[x]$ s.t. $b$ is prime to $f/d$, strictly positive at the real roots of $f/d$ and $b\,d^2\equiv g \mod f$.
  \item Compute the factorization $f/d = \prod_{i=1}^r p_i^{e_{i}}$ into irreducible factors in $\Q[x]$; 
  \item For each irreducible factor $p_i$, \\ $b_i \leftarrow$ the remainder   of $b$ modulo $p_i$; 
  
 $(Q'_i,q'_i)\leftarrow$ output of Algorithm \ref{algo:strictpos} applied to $b_i$ and $p_i$;
 
    Compute $\omega_{i,j}\in \Q_+, \overline h_{i,j}\in \Q[x]$ such that 
    $$b_i \equiv  \sum_{j} \omega_{i,j} \overline h_{i,j}^2 \mod p_i;
    $$
    
    Compute $ h_{i,j}\in \Q[x]$ such that 
    $$
    b_i \equiv  \sum_{j} \omega_{i,j} h_{i,j}^2  \mod p_i^{e_i}
    $$
    
    using Hensel lifting in Lemma \ref{lem:Hensel};
\item Compute $\omega_i\in \Q_+$, $ h_i\in \Q[x]$ such that 
$$b \equiv \sum_{i} \omega_{i} \overline{h}_{i}^2  \mod {f/d}$$

using Chinese Remainder construction in Lemma \ref{lem:CRT}; 
 \item  $h_i \leftarrow d \, \overline{h}_i$;
 \item Compute $q\in \Q[x]$ s.t. $g=\sum_i\omega_ih_i^2 + q\,f$;
\end{enumerate}
\textbf{Output:} $\omega_i\in \Q_+$, $h_i\in \Q[x]$, $q\in \Q[x]$ 
satisfying
$$
g = \sum_{i} \omega_i h_i^2 + q\, f.
$$
\end{algorithm}
\subsection{Example} 
We now revisit Example~\ref{ex:toy1} to  illustrate the symbolic-numeric approach based on Semi-Definite-Programming.

\begin{example} Let $f=x^3-2=(x-{2^{1/3}})(x-2^{1/3}\omega)(x-2^{1/3}\overline\omega)$, where $\omega=e^{2\pi \ib/3}$,  and $g=x$. 

Solving the convex optimization program:
$$
\begin{array}{cl}
     \max  & \quad \lambda \\
     s.t. & Q\in S^3(\R), Q - \lambda\, I\succcurlyeq 0\\
          & q\in \R[x]_{1}\\
          & g = \xx^t Q \xx + q f
\end{array}
$$
we obtain the matrix $Q^*$ of maximal rank and the polynomial $q^*$: 
\begin{align*}
Q^* &\approx  
\left[
\begin{array}{ccc}
0.6322063 & -0.0167531 & -0.2295612 \\
-0.0167531 & 0.4591225 & -0.1580516 \\
-0.2295612 & -0.1580516 & 0.5167531 \\
\end{array}
\right], \\
q^* &\approx  
 -0.5167531\,x + 0.3161031.
\end{align*}
The eigenvalues of $Q^*$ are approximately:
$$
0.246693, 0.5292293, 0.8321596.
$$
The norm of the error polynomial is $\rho\approx 1.16e^{-15}$ so that $\delta \approx 0.0227$ and rounding with $t=2$ decimal digits yields a positivity certificate. In fact, in this case, rounding with one decimal digit is enough:

$$
\overline Q =  \left[
\begin{array}{ccc}
0.6 & 0 & -0.2 \\
0 & 0.5 & -0.2 \\
-0.2 & -0.2 & 0.5 \\
\end{array}
\right]\quad\mbox{and}\quad 
q   =  -0.5\,x + 0.3,
$$
with error $e=\xx^T\overline Q\xx + q\,f -g = -0.1 x^3+ 0.1 x^2$ yield
 $$Q=\pi_{g-qf}(\overline Q)= \overline Q-Q_e=  \left[
\begin{array}{ccc}
\frac{3}{5} & 0 & \frac{-7}{30} \\[1mm]
0 & \frac{7}{15} & \frac{-3}{20} \\[1mm]
\frac{-7}{30} & \frac{-3}{20} & \frac{1}{2} \\
\end{array}
\right]. $$
It is a positive definite matrix (its eigenvalues are  approximately $0.24507,$  $0.505399, 0.816198$) which induces a  rational SOS decomposition of $g$ modulo $f$.
\end{example}

\section{Conclusion}
In this work,
\begin{enumerate}
    \item we showed that a univariate rational polynomial $g$ is strictly positive at all the real roots of a  univariate rational squarefree polynomial $f$ if and only if it is a sum of squares of rational univariate polynomials modulo $f$. To our knowledge, this fact was known for univariate polynomials in the global setting but not in the local setting;
    \item we showed that the usual assumption of $g$ being strictly positive at the real roots of a squarefree polynomial $f$ can be relaxed to non-negative when $\gcd(f,g)$ and $f/\gcd(f,g)$ relatively prime, which we believe  is the best assumption one can obtain; 
    \item we produced an algorithm for the local setting, which is the counterpart of known algorithms for the global setting in the strictly positive case, and involves Hensel lifting and Chinese Remainder Theorem in the non squarefree and non-negative case.
\end{enumerate}
Our projects are to derive bit complexity estimates for the proposed algorithms and also  to try to extend our results to the multivariate local setting of polynomials being non-negative at the real zero set of a  zero-dimensional ideal. Some of them can be extended mutatis-mutandis but there is still work to be done on the relaxation of the assumptions.


\end{document}